\documentclass[11.5pt,psamsfonts]{amsart}
\usepackage{amsmath}
\usepackage{amsthm}
\usepackage{amssymb}
\usepackage{amscd}
\usepackage{amsfonts}
\usepackage{amsbsy}
\usepackage{graphicx}
\usepackage[dvips]{psfrag}

\newtheorem {theorem} {Theorem}
\newtheorem {proposition} [theorem]{Proposition}
\newtheorem {corollary} [theorem]{Corollary}
\newtheorem {lemma}  [theorem]{Lemma}

\newtheorem {remark} [theorem]{Remark}

\begin{document}

\title[Perturbation of    $Q_4$]
{On the number of  limit cycles in quadratic  perturbations of quadratic codimension four   centers}

\author[ Yulin Zhao]{ Yulin Zhao}

\address{ Department of Mathematics, Sun Yat-sen University, Guangzhou, 510275, People's Republic of China.}

\email{mcszyl@mail.sysu.edu.cn}

\thanks{Supported by  NSF of China (No. 10871214) and Program for New Century Excellent Talents in University.}


\date{}
\dedicatory{}

\maketitle

\maketitle

\begin{abstract}
This paper is concerned with the bifurcation of limit cycles in
general quadratic perturbations of quadratic codimension-four
centers $Q_4$. Gavrilov and Iliev set an upper bound of {\it eight}
for the number of limit cycles produced from the period annulus
around the center.  Based on Gavrilov-Iliev's proof, we prove in
this paper  that the perturbed system has at most five  limit cycles
which emerge from the period annulus around the center. We also show
that there exists a perturbed system with three limit cycles
produced by the period annulus of $Q_4$.
\end{abstract}

\section{Introduction and statement of the main results}
In this paper we study the bifurcation of limit cycles in plane quadratic systems under small quadratic perturbations. We assume that
the unperturbed system has at least one center. Taking a complex coordinate $z=x+iy$ and using the terminology from \cite{Zo}, the list
of quadratic centers at $z=0$ looks as follows:
\begin{equation*}
\begin{array}{ll}
\dot z=-iz-z^2+2|z|^2+(b+ic){\bar z}^2,&\mathrm{Hamiltonian}\,(Q_3^H),\\
\dot z=-iz+az^2+2|z|^2+b{\bar z}^2,&\mathrm{reversible}\,(Q_3^R),\\
\dot z=-iz+4z^2+2|z|^2+(b+ic){\bar z}^2,&|b+ic|=2,\,\mathrm{codimension\,\, four}\,(Q_4),\\
\dot z=-iz+z^2+(b+ic){\bar z}^2,&\mathrm{generalized\,\, Lotka-Volterra}\,(Q_3^{LV}),\\
\dot z=-iz+{\bar z}^2,&\mathrm{Hamiltonian\,\, triangle},
\end{array}
\end{equation*}
where $a,b,$ and $c$ stand for arbitrary real constant. Let
\begin{equation}\label{eq1}
\dot x=\dfrac{H_y(x,y)}{M(x,y)},\,\,\,\,\dot y=-\dfrac{H_x(x,y)}{M(x,y)},
\end{equation}
be any of the above systems rewritten in $(x,y)$ coordinates. Here $H(x,y)$ is a first integral of system
\eqref{eq1} with the integrating factor $M(x,y)$. Consider a small quadratic perturbations of \eqref{eq1}:
\begin{equation}\label{eq2}
\dot x=\dfrac{H_y(x,y)}{M(x,y)}+\epsilon X_2(x,y,\epsilon),\,\,\,\,\dot y=-\dfrac{H_x(x,y)}{M(x,y)}+\epsilon Y_2(x,y,\epsilon),
\end{equation}
where $X_2(x,y,\epsilon)$ and $Y_2(x,y,\epsilon)$ are  quadratic polynomials in $x$ and $y$ with coefficients depending
analytically on the small parameter $\epsilon$.

Each center of system \eqref{eq1} is surrounded by a continuous set of period annuli.  Compactifying   the phase plane $\mathbb{R}^2$ of system \eqref{eq1} to the
Poincar\'e disc, the boundary of the period annulus of the center has two connected
components, the center itself and a singular loop which consists orbit(s) and at least one singularity.
It is well known that the  limit cycles of system \eqref{eq2} can
emerge from

(a) the center (i.e.the inner boundary),

(b) the  singular loop (i.e.,the outer boundary),

(c) the period annulus.

Bautin\cite{B}  found that at most three limit cycles can appear near a focus or a center of any quadratic system. This implies that
the cyclicity of the center of quadratic system is equal to three under quadratic perturbation. As usual, we
use the notion of {\it cyclicity} for the total number of limit cycles which can
emerge from a configuration of trajectories (center, period annulus, a singular loop) under a perturbation.

The  bifurcation of limit cycles from saddle-loop  in perturbations of quadratic
Hamiltonian systems has been studied in \cite{HI1}. Moreover, if the loop contains only one saddle and
under certain genericity conditions, it was proved in \cite{M}
that the cyclicity  of a singular loop can be transferred to the cyclicity of the period annuli.
However, if the loop contains at least two saddles, this transfer in general is not true. For more details, we refer to
\cite{LR} and references therein.

The cyclicity of the period annulus of system \eqref{eq1}, also known as the (extended) infinitesimal 16th Hilbert problem\cite{A} for $n=2$,
was investigated by many authors.  This problem is
reduced to counting the number of zeros of the displacement function
\begin{equation}\label{eq3}
 d(h, \epsilon) = \epsilon M_1(h)+ \epsilon^2 M_2(h)+\cdots+M_k(h)+\cdots,
\end{equation}
where $d(h, \epsilon)$ is defined on a section to the flow, which is parameterized
by the Hamiltonian value $h$. The number of zeros of the first non-vanishing Melnikov function $M_k(h)$
determines the upper bound of the number of limit cycles in \eqref{eq2} emerging
from the periodic annulus of the unperturbed integrable  system \eqref{eq1}. The corresponding Melnikov functions were determined in \cite{I1}
for quadratic centers.

The cyclicity of the period annulus for quadratic Hamiltoinian
$Q_3^H$ and Hamiltoinian triangle, were completely solved by several
authors. See \cite{CLLZ,G,HI2,I1,I2,ZLL,ZZ} and references therein.
The  generalized Lotka-Volterra $Q_3^{LV}$  has been studied by
Zoladek in \cite{Zo}. Some results concerned with certain specific
case of $Q_3^R$ can be found in  \cite{CJ,DLZ,ILY,LZ} etc. However,
almost nothing is known about the generic reversible case $Q_3^R$.
Recently, the authors of the paper \cite{GGI} propose a program for
finding the cyclicity of period annuli of quadratic centers of genus
one.  Garu, Manosan and Villadelprat \cite{GMV}
 have also some new results in this direction.

 The present paper deals with the cyclicity of period annulus of quadratic codimension four centers $Q_4$. Using Picard-Fuchs equations and
 Petrove's method(based on the argument principle)\cite{P}, Gavrilov and Iliev\cite{GI} proved that the cyclicity of period annulus of $Q_4$
 is less or equal to {\it eight}, see Theorem \ref{t2} in Section 2 below. Based on their proof, we get the following theorem in this paper.
 \begin{theorem}\label{t1}
 Let system \eqref{eq1} be a quadratic codimension four system $Q_4$ rewritten in $(x,y)$ coordinates. Then the perturbed quadratic
 system \eqref{eq2} has at most five limit cycles which emerge from the period
annulus around the center. Moreover, there exists the quadratic polynomials
 $X_2(x,y,\epsilon)$ and $Y_2(x,y,\epsilon)$ such that system \eqref{eq2} has at least three limit cycles produced by the period annulus of system
 \eqref{eq1}.
\end{theorem}

The rest of this paper is organized as follows.  In Section 2 we
briefly sketch the proof of Gavrilov- Iliev's Theorem, which is
crucial for our analysis.
 In Section 3 the  explicit forms of several
related functions  are given by revisiting  Gavrilov- Iliev's proof,
and then we get the asymptotic expansions for these functions in Section 4.
Theorem \ref{t1} is proved in Section 5. Finally we give some comments in Section 6.

\section{A sketch of proof of  Gavrilov-Iliev's theorem}
 Gavrilov and Iliev proved the following theorem:
 \begin{theorem}\label{t2}\cite{GI}
 The cyclicity of the open period annulus surrounding the center of any generic codimension four plane quadratic system
 is less than or equal to eight.
 \end{theorem}
 We are going to sketch the proof of  Theorem \ref{t2}.

 It is well known that the cyclicity of period annulus can be detected in a compact region by the number of zeros of the first non-vanishing
$k$-th order Melnikov function $M_k(h)$  in \eqref{eq3}, which is sometimes called generating function.  The generating function for $Q_4$ is given in
\cite{I1} by a complete elliptic integral. After a series of changes, the generating function becomes
\begin{equation}\label{eq4}
I(h)=\mu_1hI_{0,0}+\mu_2I_{1,0}+\mu_3I_{0,1}+\mu_4(2I_{-1,0}+3\kappa hI_{-1,1}),
\end{equation}
where
\begin{equation}
I_{i,j}(h)=\iint_{H(x,y)<h}x^iy^jdxdy,\,\,\,\,\,h\in\left(-\dfrac{2}{3},-\dfrac{2}{3\sqrt{\kappa}}\right),
\end{equation}
with
\begin{equation}\label{eq6}
H(x,y)\equiv \dfrac{2}{3}(\kappa-1)x^3-(\kappa-1)x^2y+\dfrac{\kappa}{3}y^3-y=h,\,\,\,\,\,\kappa>1.
\end{equation}
The integrals $I_{i,j}$ in \eqref{eq4} satisfy the following Picar-Funchs system
\begin{equation}\label{PF1}
\begin{array}{ccl}
I_{0,0}&=&\dfrac{3h}{2}I_{0,0}'+I_{0,1}',\\[2ex]
I_{1,0}&=&hI_{1,0}'+\dfrac{2}{3}I_{1,1}',\\[2ex]
I_{0,1}&=&\dfrac{2}{3\kappa}I_{0,0}'+hI_{0,1}'+\dfrac{2(\kappa-1)}{3\kappa}I_{1,1}',\\[2ex]
I_{1,1}&=&\dfrac{3h}{8}I_{0,0}'+\dfrac{1}{2}I_{1,0}'+\dfrac{1}{4}I_{0,1}'+\dfrac{3h}{4}I_{1,1}',\\[2ex]
I_{-1,0}&=&3hI_{-1,0}'+2I_{-1,1}',\\[2ex]
I_{-1,1}&=&\dfrac{\kappa-1}{\kappa}I_{1,0}'+\dfrac{1}{\kappa}I_{-1,0}'+\dfrac{3h}{2}I_{-1,1}'.\\[2ex]
\end{array}
\end{equation}
By using the above system, the authors get
\begin{equation}\label{eq8}
\dfrac{d}{dh}\left(\dfrac{I(h)}{h}\right)=\dfrac{hI'(h)-I(h)}{h^2}=\dfrac{\bar G(h)}{h^2},
\end{equation}
where
\begin{equation}\label{G}
\bar G(h)=(\mu_1h^2+\mu_3)I_{0,0}'+\mu_2I_{1,1}'+\mu_4[-4hI_{-1,0}'+(3\kappa
h^2-4)I_{-1,1}'].
\end{equation}
Therefore
\begin{equation*}
I(h)=h\int_{-2/3}^h\xi^{-2}\bar G(\xi)d\xi
\end{equation*}
and $I(h)$ has at most as much zeros as $\bar G(h)$.  It is  proved  in \cite{GI}  that $\bar G(h)$ satisfies the following
equation
\begin{equation}\label{eq10}
L_2(h)\bar G(h)=R(h)\triangleq \dfrac{h((\bar a_0+\bar a_1 h^2+\bar a_2 h^4+\bar a_3 h^6)I_{0,0}'+(\bar b_0+\bar b_1 h^2+\bar b_2 h^4)I_{1,1}')}{(9h^2-4)^2(9\kappa h^2-4)},
\end{equation}
where
\begin{equation}\label{L2}
L_2(h)=5\kappa h-(9\kappa h^2-8)\dfrac{d}{dh}+h(9\kappa h^2-4)\dfrac{d^2}{dh^2}.
\end{equation}

Taking the changes
\begin{equation}\label{eq12}
h=-\dfrac{2}{3}\sqrt{\dfrac{s}{\kappa}},\,\,J_1(s)=I_{0,0}'(h(s)),\,\,J_2(s)=I_{1,1}'(h(s)),\,\,\,G(s)=\bar G(h(s)),
\end{equation}
we obtain the equation
\begin{equation}\label{L2G}
L_2G=\left(s(1-s)\dfrac{d^2}{ds^2}-\dfrac{1}{2}\dfrac{d}{ds}-\frac{5}{36}\right)G(s)=\dfrac{P_3(s)J_1(s)+Q_2(s)J_2(s)}{(s-\kappa)^2(s-1)},
\end{equation}
where $P_3(s)$ and $Q_2(s)$ are real polynomials of degree at most three and two. Denote by dot the differentiation with respect to $s$. The integrals
$J_1(s)$ and $J_2(s)$ satisfy the following Picard-Fuchs equation
\begin{equation}\label{PF2}
6(s-1)(s-\kappa)\left(\begin{array}{c}\dot J_1\\ \dot J_2\end{array}\right)=\left(\begin{array}{cc}1-s&\kappa-1\\ 1-s&s-1\end{array}\right)
\left(\begin{array}{c}J_1\\ J_2\end{array}\right).
\end{equation}

We say that $V$ is a Chebyshev space, provided that each non-zero function in $V$ has at most $\mathrm{dim}(V)-1$ zeros, counted with multiplicity.
\begin{proposition}\label{p3}\cite{GI}  The following statements hold:
\begin{itemize}
\item[(i)]
Suppose the solution space of the homogeneous equation $x''+a_1(t)x'+a_2(t)x=0$ is a Chebyshev space and let $R(t)$ be an analytic function
on $(a, b)$ having $k$ zeros (counted with multiplicity). Then every solution $x(t)$ of the non-homogeneous equation
\begin{equation}
x''+a_1(t)x+a_2(t)x = R(t)
\end{equation}
has at most $k +2$ zeros on $(a,b)$.
\item[(ii)] The solution space $S$ associated to the differential operator $L_2(h)$, defined in \eqref{L2},
is a Chebyshev space.
\end{itemize}
\end{proposition}

Therefore,  we firstly estimate the number of zeros of $R(h)$. Let
\begin{equation*}
V_n=\{P_nJ_1(s)+Q_{n-1}J_2(s):P_n,Q_{n-1}\in\mathbb{R}[s],\,\,\mathrm{deg}P_m,Q_m\leq m\}
\end{equation*}

\begin{proposition}\label{p4}\cite{GI}
The vector space $V_n$ is Chebyshev on the interval $(1,\kappa)$: each element has at most $\mathrm{dim} V_n - 1 = 2n$ zeros (counted
with multiplicity).
\end{proposition}

\begin{proof}[Proof of Theorem \ref{t2}]
It follows from Proposition \ref{p4} that $P_3(s)J_1+Q_2(s)J_2$ has $6$ zeros in $(1,\kappa)$, and hence $R(h)$ has $6$ zeros
in $(-2/3,-2/(3(\sqrt{\kappa})))$. Finally one gets Theorem \ref{t2}  from Proposition \ref{p3} and \eqref{eq10}.
\end{proof}

\begin{remark}\label{r5} Proposition \ref{p4} is proved by  using argument principle  in the complex domain
$\mathbb{C}\backslash(-\infty,1]$.  The function $J_1(s)$ is a
complete elliptic integral of the first kind and therefore does not
vanish. Let
\begin{equation*}
F (s) =\dfrac{P_n(s)J_1(s)+ Q_{n-1}(s)J_2(s)}{J_1(s)}.
\end{equation*}
 Along the interval $(-\infty, 1)$, the increase of the argument of
$F$ is bounded by the number of zeros of $Q_{n-1}(s)$. Hence we have
\begin{equation}\label{eq16}
\#F(s)\leq \deg P_n(s)+\#Q_{n-1}(s)+1.
\end{equation}
in the the complex domain $\mathbb{C}\backslash(-\infty,1]$, where $\# F(s)$ denotes the number of zeros of $F(s)$.
\end{remark}

\begin{remark}
In the rest of this paper we always suppose $\kappa>1$ unless the opposite is claimed. For proof's convenience  we also suppose
that  $H(x,y)$, defined in \eqref{eq6},  is a first integral of the following system
\begin{equation}\label{eq17}
\dot x=\dfrac{\partial H}{\partial y}=-1-(\kappa-1)x^2+\kappa
y^2,\,\,\,\dot y=-\dfrac{\partial H}{\partial x}=-2(\kappa-1)x(x-y).
\end{equation}
Hence the annulus $\Gamma_h=\{(x,y)|H(x,y)=h\}$ has the negative (clockwise) orientation for $h\in(-2/3,-2/(3\sqrt{\kappa}))$. The Hamiltonian value
$h=-2/3$ and $h=-2/(3\sqrt{\kappa})$ correspond the center $(1,1)$ and the homoclinic loop $\Gamma_{-2/(3\sqrt{\kappa})}$ respectively.
\end{remark}
\section{Some paramilitary   results}

As in the paper \cite{GI}, we introduce the variable $s\in(1,\kappa)$, defined in \eqref{eq12}, and denote by dot the differentiation with respect
to $s$. Taking the changes \eqref{eq12} and
\begin{equation}\label{eq18}
J_3(s)=I_{-1,0}'(h(s)),\,\,J_4(s)=I_{-1,1}'(h(s)),\,\,J_5(s)=I_{1,0}'(h(s)),J_6(s)=I_{0,1}'(h(s)),
\end{equation}
it follows from \eqref{eq4} and \eqref{PF1}  that
\begin{equation}
\begin{array}{lcl}\label{Is}
I(s)=I(h(s))&=&\left(\dfrac{2s}{3\kappa}\mu_1+\dfrac{2}{3\kappa}\mu_3\right) J_1
+\left(\dfrac{2}{3}\mu_2+\dfrac{2(\kappa-1)}{3\kappa}\mu_3\right)J_2\\[2ex]
&&-6\mu_4\sqrt{\dfrac{s}{\kappa}}J_3+(4+2s)\mu_4J_4\\[2ex]
&&-\dfrac{2}{3}\sqrt{\dfrac{s}{\kappa}}\left(\mu_2+3(\kappa-1)\mu_4\right)J_5-\dfrac{2}{3}\sqrt{\dfrac{s}{\kappa}}(\mu_1+\mu_3)J_6.
\end{array}
\end{equation}

Suppose that $I(h)$ is defined as \eqref{eq4}. By direct computation we have that, $\bar G$, defined in \eqref{eq8}, has the form
\begin{equation}\label{eq20}
\bar G(h)=\left(\mu_1h^2-\dfrac{2\mu_3}{3\kappa}\right)I_{0,0}'+\left(-\dfrac{2\mu_2}{3}-\dfrac{2(\kappa-1)\mu_3}{3\kappa}\right)I_{1,1}'
+\mu_4(-4hI_{-1,0}'+(3\kappa h^2-4)I_{-1,1}').
\end{equation}
Here $\bar G(h)$ is different from the one defined in \eqref{G}. However if we take
\begin{equation*}
\widetilde \mu_3=-\dfrac{2\mu_3}{3\kappa},\,\,\widetilde \mu_2=-\dfrac{2\mu_2}{3}-\dfrac{2(\kappa-1)\mu_3}{3\kappa}
\end{equation*}
and omit  the tildes, then we get $\bar G(h)$, defined in \eqref{G}.

For convenience, in what follows we always suppose $\bar G(h)$ is
defined in \eqref{G} unless the opposite is claimed.

We note that $R(h)$, defined in \eqref{L2}, has no explicit form in \cite{GI}. Following the idea in \cite{GI}, a direct calculation then yields
\begin{equation}\label{R}
R(h)=\dfrac{2h(p(h)I_{0,0}'+q(h)I_{1,1}')}{(9h^2-4)^2(9\kappa
h^2-4)},
\end{equation}
where
\begin{equation}\label{pq}
p(h)=(9\kappa
h^2-4)(a_0+a_1h^2+a_2h^4),\,\,q(h)=b_0+b_1h^2+b_2h^4,
\end{equation}
with
\begin{equation}\label{ab}
\begin{array}{ccl}\
a_0&=&64\mu_1+24\mu_2+8(5\kappa+3)\mu_3+32(\kappa-1)\mu_4,\\[2ex]
a_1&=&40(\kappa-9)\mu_1-162\mu_2-18(11\kappa+9)\mu_3+24(\kappa-1)(5\kappa-9)\mu_4,\\[2ex]
a_2&=&18(\kappa+21)\mu_1+243\kappa\mu_2+486\kappa\mu_3+54\kappa(\kappa-1)\mu_4,\\[2ex]
b_0&=&-32((5\kappa-3)\mu_2+3(\kappa-1)\mu_3-4(\kappa-1)\mu_4),\\[2ex]
b_1&=&72(4(\kappa-1)\mu_1+(5\kappa^2+8\kappa-9)\mu_2+3(\kappa-1)(2\kappa+3)\mu_3\\[2ex]
&&+4(\kappa-1)(\kappa-3)\mu_4),\\[2ex]
b_2&=&-54(4(\kappa-1)(2\kappa+1)\mu_1+3\kappa(11\kappa-9)\mu_2+36\kappa(\kappa-1)\mu_3\\[2ex]
&&+12\kappa(\kappa-1)(2\kappa-3)\mu_4).
\end{array}
\end{equation}
Taking the changes \eqref{eq12}, the equation \eqref{L2G} becomes
\begin{equation}\label{eq25}
L_2G=\left(s(1-s)\dfrac{d^2}{ds^2}-\dfrac{1}{2}\dfrac{d}{ds}-\frac{5}{36}\right)G(s)=\dfrac{\kappa\mathcal{F}(s)}{1152(s-\kappa)^2(s-1)},
\end{equation}
where
\begin{equation}\label{F}
{\mathcal F}(s)=P_3(s)J_1(s)+Q_2(s)J_2(s)
\end{equation}
with
\begin{equation}\label{p3q2}
P_3(s)=p\left(-\dfrac{2}{3}\sqrt{\dfrac{s}{\kappa}}\right),\,\,Q_2(s)=q\left(-\dfrac{2}{3}\sqrt{\dfrac{s}{\kappa}}\right).
\end{equation}
Here $P_3(s)$ and $Q_2(s)$ are polynomials in $s$ with $\deg P_3(s)\leq 3,\,\,\deg Q_2(s)\leq 2$. It follows from \eqref{PF2} that
\begin{eqnarray*}
{\mathcal F}'(s)&=& P_2(s)J_1(s)+ Q_1(s)J_2(s),
\end{eqnarray*}
where
\begin{eqnarray*}
P_2(s)&=&P_3'(s)-\dfrac{P_3(s)+Q_2(s)}{6(s-\kappa)}=\alpha_0+\alpha_1 s+\alpha_2 s^2 ,\\[2ex]
 Q_1(s)&=&Q_2'(s)+\dfrac{(\kappa-1)P_3(s)+(s-1)Q_2(s)}{6(s-\kappa)(s-1)}=\beta_1 s-\beta_0,
\end{eqnarray*}
with
\begin{eqnarray*}
\alpha_0&=&\dfrac{128(42 + 13\kappa)}{9 \kappa}\mu_1+\dfrac{16(54 + 13\kappa)}{3\kappa}\mu_2+
\dfrac{32(27 + 38\kappa + 15\kappa^2)}{3\kappa}\mu_3\\[2ex]
&&-\dfrac{128(\kappa-1)(2\kappa-9)}{3\kappa}\mu_4,\\[2ex]
\alpha_1&=&\dfrac{128( -117 - 265\kappa + 30\kappa^2) }{27\kappa^2}\mu_1+\dfrac{16( -174 + 5\kappa) }{3\kappa}\mu_2
-\dfrac{16( 243 + 121\kappa) }{3\kappa}\mu_3\\[2ex]
&&+\dfrac{64( -1 + \kappa)( -119 + 60\kappa) }{9\kappa}\mu_4,\\[2ex]
\alpha_2&=&\dfrac{1088( 21 + \kappa) }{27\kappa^2}\mu_1+\dfrac{544}{\kappa}\mu_2+\dfrac{1088}{\kappa}\mu_3
+\dfrac{1088( -1 + \kappa) }{9\kappa}\mu_4,\\[2ex]
\beta_0&=&-\dfrac{256( -1 + \kappa) }{3\kappa}\mu_1-\dfrac{32( -27 + 25\kappa + 15\kappa^2) }{3\kappa}\mu_2
-\dfrac{16( \kappa-1)( 54 + 31\kappa) }{3\kappa}\mu_3\\[2ex]
&&-\dfrac{64(\kappa-1)( -18 + 5\kappa) }{3\kappa}\mu_4,\\[2ex]
\beta_1&=&-\dfrac{64(\kappa-1)( 18 + 77\kappa) }{27\kappa^2}\mu_1-\frac{16( -111 + 137\kappa) }{3\kappa}\mu_2-\dfrac{768(\kappa-1) }{\kappa}\mu_3\\
&&-\dfrac{64( -1 + \kappa)( -116 + 77\kappa) }{9\kappa}\mu_4.
\end{eqnarray*}
Solving $\mu_i,\,i=1,2,3,4$, from the above equations, we get
\begin{equation*}
\begin{array}{ccl}
\mu_1&=&\dfrac{9(162 - 213\kappa - 205\kappa^2) }{246400}\alpha_1
+\dfrac{9( 23328 + 17604\kappa - 79904\kappa^2 - 59165\kappa^3) }{54454400}\alpha_2\\[2ex]
&&-\dfrac{9(27-18\kappa - 5\kappa^2 ) }{22400( -1 +\kappa) }\beta_0
-\dfrac{9( 162 - 4236\kappa + 657\kappa^2 + 2845\kappa^3 ) }{3203200( -1 + \kappa) }\beta_1,\\[2ex]
\mu_2&=&-\dfrac{3( 54 - 77\kappa + 82\kappa^2 ) }{30800\kappa}\alpha_1
-\dfrac{3( 3888 + 2502\kappa - 5184\kappa^2 + 11833\kappa^3 ) }{3403400\kappa}\alpha_2\\[2ex]
&&+\dfrac{3( 9 - 7\kappa + 2\kappa^2 ) }{2800\kappa(\kappa-1) }\beta_0
+\dfrac{3( 27 - 709\kappa + 965\kappa^2 - 569\kappa^3) }{200200\kappa(\kappa-1)}\beta_1,\\[2ex]
\mu_3&=&\dfrac{3( -54 + 71\kappa + 205\kappa^2 ) }{61600\kappa}\alpha_1
-\dfrac{3( 7776 + 5868\kappa - 47598\kappa^2 - 59165\kappa^3 ) }{13613600\kappa}\alpha_2\\[2ex]
&&+\dfrac{3( 9 - 6\kappa -5\kappa^2 ) }{5600\kappa( \kappa-1)}\beta_0
+\dfrac{3( 54 - 1412\kappa - 1201\kappa^2 + 2845\kappa^3 ) }{800800\kappa( \kappa-1)}\beta_1,\\[2ex]
\mu_4&=&\dfrac{3( 486 - 1017\kappa + 90\kappa^2 + 205\kappa^3 ) }{246400\kappa( -1 + \kappa)}\alpha_1
\\[2ex]
&&+\dfrac{3( 69984 - 162\kappa- 276246\kappa^2 + 44405\kappa^3 + 59165\kappa^4 ) }{54454400\kappa(\kappa-1)}\alpha_2\\[2ex]
&&-\dfrac{3(\kappa-3)( -27 + 30\kappa + 5\kappa^2 ) }{22400\kappa(\kappa-1) ^2}\beta_0
\\[2ex]
&&-\dfrac{3( 486 - 13086\kappa+ 16683\kappa^2 + 1050\kappa^3 -2845\kappa^4 ) }{3203200\kappa( \kappa-1)^2}\beta_1,
\end{array}
\end{equation*}
which yields that
\begin{equation}\label{eq28}
\alpha_0=\beta_0-\kappa\beta_1-\kappa \alpha_1-\kappa^2 \alpha_2.
\end{equation}

\begin{remark}\label{r7}
For proof's convenience  in the rest of this paper we  also take
$\alpha_i,i=1,2,$ and $\beta_i,i=0,1,$ as the new parameters,
instead of $\mu_i,\,\,i=1,2,3,4$.

 Without loss of generality suppose $\beta_1=1$ if $\beta_1\not=0$. That is to say, $\beta_1\in\{0,1\}$.
\end{remark}

Let
\begin{equation}\label{eq29}
g(s)=\dfrac{\mathcal{F}'(s)}{J_1(s)}=P_2(s)+Q_1(s)w(s),
\end{equation}
where
\begin{equation}\label{w}
w(s)=\dfrac{J_2(s)}{J_1(s)},\,\,P_2(s)=\alpha_2s^2+\alpha_1 s+\alpha_0,\,\,Q_1(s)=\beta_1s-\beta_0,
\end{equation}
and $\alpha_0$ is defiend in  \eqref{eq28}, $\beta_1\in\{0,1\}$.

\section{Asymptotic expansions for the related functions}

In this section we are going to give the asymptotic expansions of the related functions  near the endpoints of their domain of definition.

\begin{lemma}\label{l8}
$J_i(s),\,i=1,2,\cdots,6$, have the  following asymptotic expansions near $s=\kappa$:
\begin{eqnarray*}
J_1(s)&=&J_1(\kappa)\left(1-\dfrac{5(s-\kappa)}{36(\kappa-1)}+\dfrac{385(s-\kappa)^2}{5184(\kappa-1)^2}-\dfrac{85085(s-\kappa)^3}{1679616(\kappa-1)^3}
\right.\\[2ex]
&&\left.+\dfrac{37182145(s-\kappa)^4}{967458816(\kappa-1)^4}+\cdots\right),\\[2ex]
J_2(s)&=&J_1(\kappa)\left(1+\dfrac{(s-\kappa)}{36(\kappa-1)}-\dfrac{35(s-\kappa)^2}{5184(\kappa-1)^2}+\dfrac{5005(s-\kappa)^3}{1679616(\kappa-1)^3}
\right.\\[2ex]
&&\left.-\dfrac{1616615(s-\kappa)^4}{967458816(\kappa-1)^4}+\cdots\right),\\[2ex]
J_3(s)&=&J_1(\kappa)\left(1-\dfrac{17(s-\kappa)}{36(\kappa-1)}+\dfrac{(1837\kappa-36)(s-\kappa)^2}{5184\kappa(\kappa-1)^2}
\right.\\[2ex]
&&-\dfrac{(5832 - 21276\kappa+
          496709\kappa^2)(s-\kappa)^3}{1679616\kappa^2(\kappa-1)^3}\\[2ex]
&&\left.+\dfrac{5(-419904 + 1870128\kappa - 3388824\kappa^2 + 50126789\kappa^3)(s-\kappa)^4}{967458816\kappa^3(\kappa-1)^4}+\cdots\right),\\[2ex]
J_4(s)&=&J_1(\kappa)\left(1-\dfrac{(5\kappa+12)(s-\kappa)}{36\kappa(\kappa-1)}+\dfrac{(-432 + 1848\kappa + 385\kappa^2)(s-\kappa)^2}
{5184\kappa^2(\kappa-1)^2}\right.\\[2ex]
&&-\dfrac{(69984 - 286416\kappa + 612612\kappa^2 +
          85085\kappa^3)(s-\kappa)^3}{1679616\kappa^3(\kappa-1)^3}+\\[2ex]
&&\dfrac{(37182145\kappa^4+ 356948592\kappa^3- 250327584\kappa^2+ 122332032\kappa -25194240)(s-\kappa)^4}{967458816\kappa^4(\kappa-1)^4}\\[2ex]
&&\cdots),\\[2ex]
J_5(s)&=&J_1(\kappa)\left(1+\dfrac{s-\kappa}{36(\kappa-1)}+\dfrac{(36 - 71\kappa)(s-\kappa)^2}{5184\kappa(\kappa-1)^2}\right.\\[2ex]
&&+\dfrac{(5832 - 15444\kappa +
          14617\kappa^2)(s-\kappa)^3}{1679616\kappa^2(\kappa -1)^3}\\[2ex]
&&\left.-\dfrac{5(-419904 + 1504656\kappa - 1965816\kappa^2 + 1204387\kappa^3)
          (s-\kappa)^4}{967458816\kappa^3(\kappa-1)^4}+\cdots\right),\\[2ex]
J_6(s)&=&J_1(\kappa)\left(1+\dfrac{(\kappa-6)(s-\kappa)}{36\kappa(\kappa-1)}-\dfrac{(216 - 672\kappa +
          71\kappa^2)(s-\kappa)^2}{5184\kappa^2(\kappa-1)^2}\right.\\[2ex]
          &&+\dfrac{(-34992 + 121176\kappa - 185886\kappa^2 + 14617\kappa^3)(s-\kappa)^3}{1679616\kappa^3(\kappa-1)^3}\\[2ex]
         &&\left.-\dfrac{5(2519424 - 10917504\kappa + 18833472\kappa^2 - 19076208\kappa^3 +
              1204387\kappa^4)(s-\kappa)^4}{967458816\kappa^4(\kappa-1)^4}+\cdots \right).
\end{eqnarray*}
\end{lemma}
\begin{proof}
Differentiating both sides of system \eqref{PF1}, we  have
\begin{equation}\label{eq20}
\begin{array}{ccc}
\dfrac{3h}{2}I_{0,0}''+I_{0,1}''&=&-\dfrac{1}{2}I_{0,0}',\\[2ex]
hI_{1,0}''+\dfrac{2}{3}I_{1,1}''&=&0,\\[2ex]
\dfrac{2}{3\kappa}I_{0,0}''+hI_{0,1}''+\dfrac{2(\kappa-1)}{3\kappa}I_{1,1}''&=&0,\\[2ex]
\dfrac{3h}{8}I_{0,0}''+\dfrac{1}{2}I_{1,0}''+\dfrac{1}{4}I_{0,1}''+\dfrac{3h}{4}I_{1,1}''&=&-\dfrac{3}{8}I_{0,0}'+\dfrac{1}{4}I_{1,1}',\\[2ex]
3hI_{-1,0}''+2I_{-1,1}''&=&-2I_{-1,0}',\\[2ex]
\dfrac{\kappa-1}{\kappa}I_{1,0}''+\dfrac{1}{\kappa}I_{-1,0}''+\dfrac{3h}{2}I_{-1,1}''&=&-\dfrac{1}{2}I_{-1,1}'.\\[2ex]
\end{array}
\end{equation}
Solving  $I_{i,j}''(h)$ from \eqref{eq20}, one gets
\begin{equation}\label{eq21}
(9h^2-4)(9\kappa h^2-4)\left(\begin{array}{c}I_{0,0}''\\I_{1,1}''\\I_{-1,0}''\\I_{-1,1}''\\I_{1,0}''\\I_{0,1}''\end{array}\right)
={\mathbf M}
\left(\begin{array}{c}I_{0,0}'\\I_{1,1}'\\I_{-1,0}'\\I_{-1,1}'\end{array}\right),
\end{equation}
where
\begin{equation*}
{\mathbf M}=
\left(\begin{array}{cccc}-3h(9\kappa h^2-4)&12(\kappa-1)h&0\\-3h(9\kappa h^2-4)&3h(9\kappa h^2-4)&0\\
8(\kappa-1)&-8(\kappa-1)&-6\kappa h(9h^2-4)&2\kappa (9h^2-4)\\
-12(\kappa-1)h&12(\kappa-1)h&4(9h^2-4)&-3\kappa h(9h^2-4)\\
2(9\kappa h^2-4)&-2(9\kappa h^2-4)&0&0\\
2(9\kappa h^2-4)&-18(\kappa-1)h^2&0&0
\end{array}\right).
\end{equation*}
It follows \eqref{eq21}, \eqref{eq12} and \eqref{eq18} that
\begin{equation}\label{eq23}
6(s-1)(s-\kappa)\dfrac{d}{ds}\left(\begin{array}{c} J_1\\  J_2\\ J_3\\  J_4\\  J_5\\  J_6\end{array}\right)
=\mathbf{M}^*
\left(\begin{array}{c} J_1\\  J_2\\ J_3\\  J_4\\  J_5\\  J_6\end{array}\right),\end{equation}
with
\begin{equation*}
\mathbf{M}^*=\left(\begin{array}{cccc} 1-s&\kappa-1&0&0\\1-s &s-1&0&0\\
-\dfrac{\sqrt{\kappa}(\kappa-1)}{\sqrt{s}} &\dfrac{\sqrt{\kappa}(\kappa-1)}{\sqrt{s}}&-2(s-\kappa)&-\dfrac{\sqrt{\kappa}(s-\kappa)}{\sqrt{s}}
\\ 1-\kappa &\kappa-1&-\dfrac{2(s-\kappa)}{\sqrt{s\kappa}}&\kappa-s
\\ -\dfrac{\sqrt{\kappa}(s-1)}{\sqrt{s}} &\dfrac{\sqrt{\kappa}(s-1)}{\sqrt{s}}&0&0\\
-\dfrac{\sqrt{\kappa}(s-1)}{\sqrt{s}} &\dfrac{(\kappa-1)s}{\sqrt{s\kappa}}&0&0\end{array}\right).
\end{equation*}
Since $h=-2/3$ corresponds the center $(1,1)$ of Hamiltonian system
\eqref{eq17}, we have $I_{i,j}(-2/3)=0,\,\,i=0,\pm 1,\,\,j=0,\pm 1$.
Substituting $I_{i,j}(-2/3)=0$ into \eqref{PF1} yields
$I_{-1,1}'(-2/3)=I_{-1,0}'(-2/3)=I_{1,0}'(-2/3)=I_{0,1}'(-2/3)=I_{1,1}'(-2/3)=I_{0,0}'(-2/3)$.
Therefore $J_1(\kappa)=J_2(\kappa) =\cdots=J_6(\kappa)$.

Since $s=\kappa$ corresponds to the center of Hamiltonian system
\eqref{eq17}, $J_i(s)$ is analytic at $s=\kappa$. Taking
$J_i(s)=\sum_{j=0}^{\infty}c_{i,j}(s-\kappa)^j$ with
$c_{i,0}=J_1(\kappa)$ into \eqref{eq23}, we get the expansions.
\end{proof}

\begin{corollary}\label{c9}
 The following assertions hold:
\begin{itemize}
\item[(i)] $I(s)\equiv 0$ if and only if $\mu_1=\mu_2=\mu_3=\mu_4=0$.
\item[(ii)] There exists $\mu_i,\,\,i=1,2,3,4$, such that $I(s)$ has at most three zeros in $(1,\kappa)$.
\end{itemize}
\end{corollary}
\begin{proof}
It follows from \eqref{Is} and Lemma \ref{l8} that $I(s)$ has the
following asymptotic expansion at $s=\kappa$:
\begin{equation}
I(s)=\nu_1(s-\kappa)+\nu_2(s-\kappa)^2+\nu_3(s-\kappa)^3+\nu_4(s-\kappa)^4+\cdots,
\end{equation}
where $\nu_i=d^{i}I(\kappa)/ds^i,\,i=1,2,3,4$, with
\begin{equation}\label{eq35}
\begin{array}{ccl}
\nu_1&=&\dfrac{2\mu_1}{9\kappa}-\dfrac{\mu_2}{3\kappa}-\dfrac{\mu_3}{3\kappa}+\dfrac{2(\kappa-1)\mu_4}{3\kappa},\\[2ex]
\nu_2&=&\dfrac{(13\kappa-18)\mu_1}{162\kappa^2(\kappa-1)}+\dfrac{(17\kappa-18)\mu_2}{108\kappa^2(\kappa-1)}
+\dfrac{(17\kappa-12)\mu_3}{108\kappa^2(\kappa-1)}+\dfrac{(13\kappa+18)\mu_4}{54\kappa^2},\\[2ex]
\nu_3&=&-\dfrac{(1944 - 4068\kappa +
          1739\kappa^2)\mu_1}{11664\kappa^3(\kappa-1)^2}-\dfrac{(1944 - 3780\kappa +
          1801\kappa^2)\mu_2}{(7776\kappa^3(\kappa-1)^2}\\[2ex]
     &&-\dfrac{(1296 - 2712\kappa +1801\kappa^2)\mu_3}{7776\kappa^3(\kappa-1)^2}-\dfrac{(-1944+ 1260\kappa+
          1739\kappa^2)\mu_4}{3888\kappa^3(\kappa-1)^2},\\[2ex]
\nu_4&=&\dfrac{5(-104976 + 324648\kappa - 338526\kappa^2 + 101837\kappa^3)\mu_1}{1259712\kappa^4(\kappa-1)^3}\\[2ex]
&&+\dfrac{5(-104976 + 309096\kappa - 301374\kappa^2 + 96253\kappa^3)\mu_2}{839808\kappa^4(\kappa-1)^3}\\[2ex]
&&+\dfrac{5(-69984 + 216432\kappa - 225684\kappa^2 + 96253\kappa^3)\mu_3}{839808\kappa^4(\kappa-1)^3}\\[2ex]
&&+\dfrac{5(104976 - 180792\kappa + 7398\kappa^2 + 101837\kappa^3)\mu_4}{419904\kappa^4(\kappa-1)^3}.
\end{array}
\end{equation}
System \eqref{eq35} is a linear system of  equations in the variables $\mu_i,\,\,i=1,2,3,4$.
  The determinant of matrix of coefficients of \eqref{eq35} is equal to $125/(472392\kappa^8(\kappa-1)^4)>0$
for $\kappa>1$. As shown by Cramer's rule, system \eqref{eq35} has  a unique solution. Therefore $\nu_1=\nu_2=\nu_3=\nu_4=0$ if and only if
$\mu_1=\mu_2=\mu_3=\mu_4=0$, which yields that  $I(s)\equiv 0$ if and only if $\mu_i=0,\,\,i=1,2,3,4$.
This proves the assertion (i).

Since system \eqref{eq35} has  a unique solution,  we can choose
$\nu_i$ as the independent  parameters, instead of
$\mu_i,\,\,i=1,2,3,4$. Denote by $
I(s,\nu_1,\nu_2,\nu_3,\nu_4)=I(s).$ Without loss of generality
suppose  $\nu_4>0$. To get more zeros of $I(s)$, we choose $\nu_i$
and $s_i\in(1,\kappa),i=4,3,2,1,$  such that
$I(s_4,0,0,0,\nu_4)>0,\,\,I(s_3,0,0,\nu_3,\nu_4)<0,\,\,I(s_2,0,\nu_2,\nu_3,\nu_4)>0,\,\,I(s_1,\nu_1,\nu_2,\nu_3,\nu_4)<0$
and $0<|\nu_1|\ll|\nu_2|\ll|\nu_3|\ll|\nu_4|$,
$1<s_4<s_3<s_2<s_1<\kappa$. It is easy to show that $I(s)$, defined
as the above, has at least three zeros which tend to $\kappa$.
\end{proof}
\vskip 0.3cm
\begin{lemma}\label{l10}
$J_i(s),\,i=1,2,\cdots,6$, have the following asymptotic expansions near $s=1$:
\begin{eqnarray*}
J_1(s)&=&-\dfrac{\ln(s-1)}{2\sqrt{\kappa-1}}+c_{12}-\dfrac{5(s-1)\ln(s-1)}{72(\kappa-1)^{3/2}}\cdots,\\[2ex]
J_2(s)&=&\dfrac{3}{\sqrt{\kappa-1}}-\dfrac{(s-1)\ln(s-1)}{12(\kappa-1)^{3/2}}+\cdots,
\end{eqnarray*}
where $c_{12}$ is a real  constant.
\end{lemma}
\begin{proof}
Since the period annuli of the vector field \eqref{eq17}  begin at
the center $(1,1)$ and terminate at a homoclinic loop
$\Gamma_{-2/(3\sqrt{\kappa})}=\{(x,y)|H(x,y)=-2/(3\sqrt{\kappa})\}$,
it follows from \cite{R} that $I_{i,j}(h),\,\,i\geq 0$, have the
asymptotic expansions of the form
\begin{equation*}
I_{i,j}(h)=\sum_{k=0}^{\infty}d_{i,j,k}\left(-\dfrac{2}{3\sqrt{\kappa}}-h\right)^k
+\ln\left(-\dfrac{2}{3\sqrt{\kappa}}-h\right)\sum_{k=1}^{\infty}{\tilde d}_{i,j,k}\left(-\dfrac{2}{3\sqrt{\kappa}}-h\right)^k,
\end{equation*}
as $h\rightarrow -2/(3\sqrt{\kappa})$, which implies that
$J_i(s),\,\,i=1,2$, have the asymptotic expansions of the form
\begin{equation}\label{eq36}
J_i(s)=c_{i,1}\ln(s-1)+c_{i,2}+c_{i,3}(s-1)\ln(s-1)+\cdots
\end{equation}
as $s\rightarrow 1$. On the other hand it is well known (see for instance \cite{Li,R} or the appendix of \cite{ZZ}) that
\begin{equation*}
\tilde{d}_{i,j,1}=\dfrac{x^iy^j}{2\sqrt{\kappa-1}}\Bigg|_{(x,y)=(0,1/\sqrt{\kappa})}=\left\{\begin{array}{cl}
\dfrac{1}{2\sqrt{\kappa^j(\kappa-1)}},&\mathrm{if}\,\,i=0,\\[2ex]
0,&\mathrm{if}\,\,i=1.
\end{array}
\right.
\end{equation*}
A simple calculation shows that $c_{i,1}=-{\title d}_{i,j,1}$.
Taking \eqref{eq36} with $c_{i,1}=-{\title d}_{i,j,1}$ into
\eqref{PF2}, we obtain the asymptotic expansions near $s=1$ for
$J_i(s),\,\,i=1,2$.
\end{proof}
\vskip 0.3cm By Lemma \ref{l8} and Lemma \ref{l10}, a straight
calculation shows that the following two lemmas hold.

\begin{lemma}\label{l11}
The following expressions hold:
\begin{equation*}w(\kappa)=1,\,\,w'(\kappa)=\dfrac{1}{6(\kappa-1)},\,\,w''(\kappa)=-\dfrac{25}{216(\kappa-1)^2},\,\,w'''(\kappa)=\dfrac{775}{3888(\kappa-1)^3},\\[2ex]
\end{equation*}
\begin{equation*}
g(\kappa)=0,\,\,{\mathcal F}(\kappa)={\mathcal
F}'(\kappa)=0,\,\,{\mathcal F}''(\kappa)=J_1(\kappa)g'(\kappa).
\end{equation*}
\end{lemma}

Since $G(s)$ is analytic at $s=\kappa$, it follows from \eqref{eq25} that $\mathcal {F}(\kappa)=\mathcal {F}'(\kappa)=0$. This is verified by Lemma
\ref{l11}.
\begin{lemma}\label{l12}
The the following expansions holds  as $s\rightarrow 1$:
\begin{eqnarray*}
w(s)&=&-\dfrac{6}{\ln(s-1)}+\cdots,\\[2ex]
g(s)&=&\left\{\begin{array}{ll}P_2(1)-\dfrac{6Q_1(1)}{\ln(s-1)}+\cdots,&\mathrm{if}\,\,Q_1(1)\not=0,\\[2ex]
P_2(s),&\mathrm{if}\,\,\beta_0=\beta_1=0,\\[2ex]
P_2(1)-\dfrac{6(s-1)}{\ln(s-1)}+\cdots,&\mathrm{if}\,\,\beta_0=\beta_1=1,
\end{array}\right.
\end{eqnarray*}
where  $\beta_1\in\{0,1\}$.
\end{lemma}

\section{Proof of Theorem \ref{t1}}

First of all we note that the following proposition holds:

\vskip0.3cm
\begin{proposition}\label{p13}
Denote by  $\#{\mathcal F}(s)$  the number of zeros of ${\mathcal
F}(s)$. Then we have
\begin{equation}\label{eq35.2}
\#\mathcal F(s)\leq \#\mathcal{F}'(s)=\# g(s),\,\,\# I(s)\leq
\#G(s)\leq \#\mathcal{F}(s)+2.
\end{equation}
\end{proposition}
\begin{proof}
It follows from Lemma \ref{l11} that ${\mathcal F}(\kappa)={\mathcal
F}'(\kappa)=0$ This yields the first inequality of \eqref{eq35.2}.
The second inequality is obtained by  Proposition \ref{p3}.
\end{proof}

 Noting that the cyclicity of period annulus is determined by $\#I(s)$, we will prove Theorem \ref{t1} by estimating the number
of zeros of $g(s)$.

Since $g(\kappa)=0$, we get that  $g(s)$ has at most three zeros in
$(1,\kappa)$ by argument principle, see Remark \ref{r5}. This
implies that $\#I(s)\leq \#G(s)\leq 5$. However, to get more
information about the number of zeros of $g(s)$ (hence $I(s)$), we
prefer to prove Theorem \ref{t1} by the following theorem, see the
comments in the next section, and the note after the statement of
this theorem.

\begin{theorem}\label{t14}
Let $s\in(1,\kappa)$ and $\beta_1\in\{0,1\}$.
\begin{itemize}
\item[(a)] If $\beta_1=1,\,\,\beta_0\in(-\infty,(23\kappa-54)/31]\cup[1,+\infty)$, then $g(s)$ has at most two zeros
 in $(1,\kappa)$.
\item[(b)] If $\beta_1=1,\,\,\beta_0\in((23\kappa-54)/31,1)$, then $g(s)$ has at most three zeros.
\item[(c)] If $\beta_1=0,\,\,\beta_0\not=0$, then  $g(s)$ has at most two zeros in the same interval.
\item[(d)] If $\beta_1=0,\,\,\beta_0=0$, then $g(s)$ has at most one zeros.
\end{itemize}
\end{theorem}

If $\beta_1=1,\,\,\beta_0\in[1,+\infty)$ (resp. $\beta_0\in(-\infty,1)$), then it can be proved that $g(s)$ has at most two (resp. three)
 zeros in $(1,\kappa)$ by argument principle, see
the proof of Proposition \ref{p4}\cite{GI}, or Remark \ref{r5}. However it seems that we can not prove by argument principle
 that $g(s)$ has at most two zeros in $(1, \kappa)$ if $\beta_1=1,\,\,\beta_0\in(-\infty,(23\kappa-54)/31]$.

\vskip 0.3cm

Firstly we study the geometric properties of $w(s)=J_2(s)/J_1(s)$.

\begin{lemma}\label{l15}
The function $w(s)$ is   monotonically  increasing and concave in
the interval $(1,\kappa)$, i.e., $w'(s)>0,\,\,w''(s)<0$ and
$0<w(s)<1$.
\end{lemma}
\begin{proof}
It follows from \eqref{PF2} that $w(s)$ satisfies
\begin{equation}\label{r2}
6(s-1)(s-\kappa)w'=1-s+2(s-1)w-(\kappa-1)w^2\triangleq U(s,w).
\end{equation}
Note that $U(s,w)$, the right hand of \eqref{r2}, is a quadratic polynomial of $w$.
 Since $4(s-1)^2+4(\kappa-1)(1-s)=4(s-1)(s-\kappa)<0$ for $s\in(1,\kappa)$ and $-(\kappa-1)<0$, we have $U(s,w)<0$. This yields
 $w'(s)>0$ for $s\in(1,\kappa)$. The inequality $0<w(s)<1$ follows from Lemma \ref{l11} and Lemma \ref{l12}.

Differentiating both sides of   \eqref{r2}, one gets
\begin{equation*}
6(s-1)(s-\kappa)w''=\left(-6(2s-\kappa-1)+\dfrac{\partial U(s,w)}{\partial w}\right)\dfrac{U(s,w)}{6(s-1)(s-\kappa)}
+\dfrac{\partial U(s,w)}{\partial s},
\end{equation*}
which implies that
\begin{equation}\label{eq39}
\dfrac{d^2w}{ds^2}=\frac{V_1(s,w)V_2(s,w)}{18(s-1)^2(s-\kappa)^2}
\end{equation}
with
\begin{equation*}
V_1(s,w)=(\kappa-1)w-(s-1),\,\,\,V_2(s,w)=(\kappa-1)w^2+(4s-3\kappa-1)w-2(s-1).
\end{equation*}
By Remark \ref{r5} (or the proof of Proposition \ref{p4}\cite{GI}),
we conclude that $(1-s)J_1(s)+(\kappa-1)J_2(s)$ has at most two
zeros in $\mathbb{C}\backslash(-\infty,1]$. Since
$(1-s)J_1(s)+(\kappa-1)J_2(s)$ has a zero at $s=\kappa$,
$(1-s)J_1(s)+(\kappa-1)J_2(s)$ has at most one zero in $(1,\kappa)$.
Noting $J_1(s)\not=0$, we know that the function
\begin{equation*}
\eta_1(s)=V_1(s,w(s))=(1-s)+(\kappa-1)\dfrac{J_2(s)}{J_1(s)}
\end{equation*}
has at most one zero in $(1,\kappa)$. By direct computation $\eta_1(\kappa)=\eta_1(1)=0$, $\lim_{s\rightarrow 1}\eta_1'(s)=+\infty$, $\eta_1'(k)=-5/6$,  which implies that the number of zeros of
$\eta_1(s)$ is even. Therefore $\eta_1(s)$ has no zero in $(1,\kappa)$, which shows that $V_1(s,w(s))>0$ in $(1,\kappa)$.

Let
\begin{equation}\label{D}
{\mathcal D}=\{(s,w)|1\leq s\leq \kappa,\,\,0\leq w\leq 1\}.
\end{equation}
Now we study the two independent variables function $V_2(s,w)$, defined in $\mathcal D$. The equations
$\partial V_2/\partial s=\partial V_2/\partial w=0$ has a unique solution at $(s,w)=((\kappa+1)/2,1/2)\in\mathcal D$. Therefore,
$V_2(s,w)$ has a maximum point  and a minimum point at either $(s,w)=((\kappa+1)/2,1/2)\in\mathcal D$, or the point in the
boundary of $\mathcal D$. Since
\begin{eqnarray*}
V_2(s,0)=-2(s-1)\leq 0,\,\,V_2(s,1)=2(s-\kappa)\leq 0,\,\,V_2(1,w)=(\kappa-1)w(w-3)\leq 0,\\[2ex]
V_2(\kappa,w)=(\kappa-1)(w-1)(w+2)\leq 0,\,\,V_2\left(\dfrac{\kappa+1}{2},\dfrac{1}{2}\right)=-\dfrac{5(\kappa-1)}{4}<0,
\end{eqnarray*}
we have $V_2(s,w)\leq 0$ for $(s,w)\in\mathcal{D}$.

Let $\eta_2(s)=V_2(s,w(s))$. Then $\eta_2(1)=\eta_2(\kappa)=0,\,\lim_{s\rightarrow 1}\eta_2'(s)=-\infty,\eta_2'(\kappa)=5/2>0$,
which implies that $V_2(s,w(s))<0$ in $(1,\kappa)$. It follows from \eqref{eq39} $w''(s)<0$ for $s\in(1,\kappa)$.
\end{proof}

Assume that  $\psi(x_1,x_2,\cdots,x_n)$ and $\phi(x_1,x_2,\cdots,x_n)$ are two polynomials in $x_1,x_2,\cdots,x_n$. Eliminating of the  variable  $x_i$ from the
equations  $\psi(x_1,x_2,\cdots,x_n)=\phi(x_1,x_2,\cdots,x_n)=0$, we get the resultant of $\psi(x_1,x_2,\cdots,x_n)$ and $\phi(x_1,x_2,\cdots,x_n)$, denoted by
$\mathrm{Resultant}(\psi,\phi,x_i)$.
\vskip 0.3cm

\begin{lemma}\label{l16}
 $w'''(s)>0,\,\,s\in(1,\kappa)$.
\end{lemma}
\begin{proof}
It follows from \eqref{eq39}  and  \eqref{r2} that
\begin{equation*}
w'''(s)=\dfrac{\Phi(s,w)}{108(s-1)^3(s-\kappa)^3},
\end{equation*}
with
\begin{eqnarray*}
\Phi(s,w)&=&-(s-1)^2(20s+\kappa-21)+2(s-1)(15-\kappa+6\kappa^2-29s-11\kappa s+20s^2)w\\[2ex]
&&-2(\kappa-1)(1+18\kappa-19s)(\kappa-s)w^2+6(\kappa-1)^2(1+3\kappa-4s)w^3\\[2ex]
&&-3(\kappa-1)^3w^4.
\end{eqnarray*}
Let $\mathcal{D}$ be the closed rectangle, defined in \eqref{D}. The maximum and minimum for   $\Phi(s,w)$ in ${\mathcal D}$ necessary occurs either  on the boundary of ${\mathcal D}$, or  the
  points inside ${\mathcal D}$ whose coordinates satisfies equations
 \begin{equation}\label{eq41}
 \Phi_s=\dfrac{\partial \Phi(s,w)}{\partial s}=0,\,\,\Phi_w=\dfrac{\partial \Phi(s,w)}{\partial w}=0.
 \end{equation}
 $\Phi_s$ and $\Phi_w$, defined in  \eqref{eq41}, are two polynomials of $s$ with the polynomial coefficients depending on $w$ and $\kappa$.
 Their resultant is
 \begin{equation*}
 \mathrm{ Resultant}\left(\Phi_s,\Phi_w,s\right)=-8000(\kappa-1)^6(w-1)^2w^2(2w-1)\chi(w)
 \end{equation*}
 with
 \begin{equation*}
\chi(w)=-160425 + 316012 w - 314956 w^2 - 2112 w^3 + 1056 w^4.
\end{equation*}
It is nice for our study that $\mathrm{ Resultant}_s\left(\Phi_s,\Phi_w,s\right)$ does not depend on $\kappa$. By Sturm Theorem
$\chi(w)$ has no real zero in $(0,1)$. Taking $w=1/2$ into the first equation of system \eqref{eq41}, we know that
$(s,w)=((k+1)/2,1/2)$ is a solution of system \eqref{eq41} in ${\mathcal D}$.
Direct computation yields that if $(s,w)\in {\mathcal D}$, then
\begin{eqnarray*}
\Phi\left(\dfrac{\kappa+1}{2},\dfrac{1}{2}\right)&=&-\dfrac{25}{16}(\kappa-1)^3<0,\\[2ex]
\Phi(1,w)&=&-3(\kappa-1)^3w^2(12-6w+w^2)\leq 0,\\[2ex]
\Phi(\kappa,w)&=&-3(\kappa-1)^3(w-1)^2(7+4w+w^2)\leq 0,\\[2ex]
\Phi(s,0)&=&-(s-1)^2(20s+\kappa-21)\leq 0,\\[2ex]
\Phi(s,1)&=&-(s-\kappa)^2(-1+21\kappa-20s)\leq 0,
\end{eqnarray*}
which imply that the maximum and minimum for $\Phi(s,w)$ are non-positive. Therefore
  $\Phi(s,w)\leq 0$ for $(s,w)\in {\mathcal D}$.

Assume that there exists the internal point $(s^*,w^*)$ of $\mathcal{D}$ such that $\Phi(s^*,w^*)=0$. Since $\Phi(s,w)\leq 0$,
 $(s^*,w^*)$ must be a maximum point of $\Phi(s,w)$ inside $\mathcal{D}$. However we have shown that the maximum and minimum for
  $\Phi(s,w)$ inside  ${\mathcal D}$ necessary occurs at $(s,w)=((k+1)/2,1/2)$ and $\Phi((\kappa+1)/2,1/2)<0$. This yields contradiction.
  Hence $\Phi(s,w)<0$ for $(s,w)\in\{(s,w)|1<s<\kappa,\,\,0<w<1\}$, which implies that $w'''(s)> 0$ for $s\in(1,\kappa)$.
\end{proof}

\begin{proposition}\label{p17}
Let $\beta_1=1$ and $s\in(1,\kappa)$. The following statements hold.
\begin{itemize}
\item[(i)] If $\beta_0\in(-\infty,(54-23\kappa)/31]$, then $g'''(s)>0$.
\item[(ii)] If $\beta_0\in((54-23\kappa)/31,1)$, then $g'''(s)$ has exactly one zero.
\item[(iii)] If $\beta_0\in[1,+\infty)$, then $g'''(s)<0$.
\end{itemize}
\end{proposition}

\begin{proof}By direct computation we have
\begin{equation*}
g'''(s)=3w''(s)+(s-\beta_0)w'''(s)=3w'''(s)\Theta(s),
\end{equation*}
 where
\begin{equation*}
\Theta(s)=\dfrac{w''(s)}{w'''(s)}+\dfrac{s-\beta_0}{3}.
\end{equation*}
Therefore,
\begin{equation}\label{eq42}
\dfrac{d\Theta(s)}{ds}=\dfrac{4(w'''(s))^2-3w''(s)w^{(4)}(s)}{3(w'''(s))^2}=
\dfrac{\Theta_1(s,w,\kappa)\Theta_2(s,w,\kappa)}{17496(s-1)^6(s-\kappa)^6w'''^2(s)}
\end{equation}
with
\begin{eqnarray*}
\Theta_1(s,w,\kappa)&=&2(s-1)^2(-9+4\kappa+5s)-(s-1)(36-67\kappa+51\kappa^2
-5s-35\kappa s\\[2ex]
&&+20s^2)w+(\kappa-1)(18-41\kappa+18\kappa^2+5s+5\kappa s-5s^2)w^2\\[2ex]
&\triangleq&\theta_0(s)+\theta_1(s)w
+\theta_2(s)w^2,\\[2ex]
\Theta_2(s,w,\kappa)&=&(s-1)^2(9+7\kappa-16s)-(s-1)(-9-62\kappa+39\kappa^2+80s+16\kappa s\\[2ex]
&&-32s^2)w-(\kappa-1)(-9+55\kappa+18\kappa^2-37s-91\kappa s+64s^2)w^2\\[2ex]
&&+9(\kappa-1)^2(1+\kappa-2s)w^3.
\end{eqnarray*}
 A calculation  shows that $\theta_2(1)=\theta_2(\kappa)=18(\kappa-1)^3>0$. Since
$\theta_2(s)$ is a quadratic polynomial in $s$ with $\lim_{s\rightarrow \pm \infty}\theta_2(s)=-\infty$, we have
 $\theta_2(s)>0$ for $s\in[1,\kappa]$. On the other hand,
 \begin{equation*}
 (\theta_1(s))^2-4\theta_0(s)\theta_2(s)=25(s-\kappa)^2(s-1)^2(81-146\kappa+81\kappa^2-16s-16\kappa s+16s^2).
 \end{equation*}
 It follows from $(16+16\kappa)^2-4(81-146\kappa+81\kappa^2)\cdot 16=-4928(\kappa-1)^2<0$ that
 $81-146\kappa+81\kappa^2-16s-16\kappa s+16s^2>0$, which implies $(\theta_1(s))^2-4\theta_0(s)\theta_2(s)>0$.

 Rewrite $\Theta_1(s,w,\kappa)$ as the form
 \begin{equation*}
\Theta_1(s,w,\kappa)=\theta_2(s)(w-W^+(s))(w-W^-(s)),
 \end{equation*}
where
\begin{equation*}
W^{\pm}(s)=\dfrac{-\theta_1(s)\pm\sqrt{(\theta_1(s))^2-4\theta_0(s)\theta_2(s)}}{2\theta_2(s)}.
\end{equation*}
This gives that
\begin{equation*}
W^+(s)=\dfrac{8(s-1)}{3(\kappa-1)}+\cdots,\,\,\,\,\,W^-(s)=\dfrac{(s-1)}{6(\kappa-1)}+\cdots,
\end{equation*}
as $s\rightarrow 1^+$, and
\begin{eqnarray*}
W^+(s)&=&1+\dfrac{(s-\kappa)}{6(\kappa-1)}-\dfrac{115(s-\kappa)^2}{324(\kappa-1)^2}+\cdots,\\[2ex]
W^-(s)&=&1+\dfrac{8(s-\kappa)}{3(\kappa-1)}+\cdots,
\end{eqnarray*}
as $s\rightarrow \kappa^-$. Therefore it follows from Lemma \ref{l11} and Lemma \ref{l12} that $W^-(s)<W^+(s)< w(s)=J_2(s)/J_1(s)$
as either $s\rightarrow 1^+$, or $s\rightarrow \kappa^-$.

 Since  $W^+(s)< w(s)$ as either $s\rightarrow 1^+$, or
$s\rightarrow \kappa^-$, $w=w(s)$ and $w=W^+(s)$ have at least two
intersection points if there exists. Suppose that $w=w(s)$
intersects  $w=W^+(s)$   at $s_1,s_2$, then $w(s_i)-W^+(s_i)=0$. By
the Mean Value Theorem there exist $s^*$ such that
$w'(s^*)-(W^+)'(s^*)=0$. Noting
$(W^\pm)'(s)=-\Theta_{1s}/\Theta_{1w}$,it follows from  \eqref{r2}
that $s=s^*$ satisfies the following equations
\begin{equation}\label{eq43}
\Theta_1(s,w,\kappa)=0,\,\,\,\widetilde{\Theta}_1(s,w,\kappa)=\Theta_{1w}U(s,w)+6(s-1)(s-\kappa)\Theta_{1s},
\end{equation}
where $\Theta_{1s}=\partial\Theta_{1}(s,w,k)/\partial s$.  Noting
that we have shown $\theta_2(s)>0$ for $s\in(1,\kappa)$, one gets
\begin{equation*}
\mathrm{Resultant}(\Theta_1,\widetilde{\Theta}_1,w)=-35083125(
\kappa-1)^5( k - s)^5( s-1)^5\theta_2(s)\not=0,\,\,s\in(1,\kappa),
\end{equation*}
which implies that two equations in \eqref{eq43} have  no common
zero. Therefore there is no $s^*\in(1,\kappa)$ such that
$w(s^*)=W^{+}(s^*)$. This yields $W^-(s)<W^{+}(s)<w(s)$ for
$s\in(1,\kappa)$. Finally we obtain that
$\Theta_1(s,w(s),\kappa)=2\theta_2(s)(w(s)-W^+(s))(w(s)-W^-(s))>0$
for $s\in(1,\kappa)$.

 Now we consider $\Theta_2(s,w,\kappa)$. Let
 \begin{equation*}
 {\mathcal D}'=\{(s,\kappa)|1\leq s\leq \kappa,\,\,1\leq \kappa\leq c\},\,\,\,c\,\,\mathrm{is\,\,a \,\, real\,\, large\,\,constant},\,\,|c|\gg
 1\},
 \end{equation*}
 be a triangle in the $\kappa s$-plane and fix $w$ as a real constant with $0<w<1$. The maximum and minimum for   $\Theta_2(s,w,\kappa)$ in ${\mathcal D}'$ necessary occurs either  on the boundary of ${\mathcal D}'$, or  the
  points inside ${\mathcal D}'$ whose coordinates satisfies equations
\begin{equation*}
\dfrac{\partial \Theta_2(s,w,\kappa)}{\partial s}=0,\,\,\dfrac{\partial \Theta_2(s,w,\kappa)}{\partial k}=0.
\end{equation*}
Direct computation shows that
\begin{equation*}
\mathrm{Resultant}(\Theta_{2s},\Theta_{2\kappa},w)=-6705(\kappa-1)^2(s-\kappa)^2(s-1)^2\gamma(s,\kappa)
\end{equation*}
with
\begin{eqnarray*}
\gamma(s,\kappa)&=&362313 + 701586 \kappa - 1012697\kappa^2 + 421884 \kappa^3 - 1012697 \kappa^4 +
    701586 \kappa^5 \\[2ex]
    &&+ 362313 \kappa^6-8(1 + \kappa)(359433 - 174116 \kappa - 174026\kappa^2 -
        174116 k^3 \\[2ex]
        && + 359433\kappa^4)s+8(991241 + 22492 \kappa - 1044426\kappa^2 + 22492k^3\\[2ex]
        &&+ 991241\kappa^4)s^2-16384(\kappa+1)(649 - 978 \kappa + 649 \kappa^2)s^3\\[2ex]
        &&+8192(809 - 658 \kappa + 809 \kappa^2)s^4-1572864 (1 +\kappa)s^5+524288s^6.
\end{eqnarray*}
Since
$\mathrm{Resultant}(\gamma_s,\gamma_\kappa,s)=c^*(\kappa-1)^{25}$
with $c^*<0$, the maximum  and  the minimum  for $\gamma(s,\kappa)$
in ${\mathcal D}'$ occurs  on the boundary of ${\mathcal D}'$.
$\gamma(s,\kappa)$ is a polynomial of $\kappa$ with degree $6$ and
the  coefficient of $\kappa^6$ is $362313$, which implies
$\gamma(s,c)>0$ as $c$ is sufficient large enough. Noting
$\gamma(1,\kappa)=\gamma(\kappa,\kappa)=362313(\kappa-1)^6$,
$\gamma(s,\kappa)$ has its minimum value $\gamma(1,1)=0$ at
$(s,\kappa)=(1,1)$. This yields $\gamma(s,\kappa)>0$ for
$(s,\kappa)\in\mathcal{D}'\backslash{(1,1)}$.

 Therefore $\mathrm{Resultant}(\Theta_{2s},\Theta_{2\kappa},w)\not=0$ for
$(s,\kappa)\in{\mathcal D}'\backslash(\{s=1\}\cup\{s=\kappa\})$,
which implies that  the maximum and minimum for
$\Theta_2(s,w,\kappa)$ in ${\mathcal D}'$ necessary occurs  on the
boundary of ${\mathcal D}'$. If $0< w< 1$ and
$(s,\kappa)\not=(1,1)$, then
 \begin{equation*}
\Theta_2(1,w,\kappa)=9(\kappa-1)^3w^2(w-2)< 0,\,\,\Theta_2(\kappa,w,\kappa)=-9(\kappa-1)^3(w-1)^2(w+1)< 0.
 \end{equation*}
Noting that $\Theta_2(1,w,\kappa)$ is a polynomial in $\kappa$ and
the coefficient of the highest order term $\kappa^3$ is
$9w^2(w-2)<0$, we have $\Theta_2(s,w,c)<0$, provided that $c$ is
sufficient large enough and $0<w<1$. Summing the above discussions
and noting $\Theta_2(1,w,1)=0$, one gets $\Theta_2(s,w,\kappa)$  has
its maximum value zero at $(s,\kappa)=(1,1)$ in ${\mathcal D}'$.
Since we always suppose that $\kappa>1$ in this paper,
$\Theta_2(s,w(s),\kappa)< 0$ for $s\in(1,\kappa)$.

It follows from \eqref{eq42} that $\Theta'(s)<0$. This yields that $g'''(s)$ has at most one zero in $(1,\kappa)$. On the other hand,
 Lemma \ref{l11} and Lemma \ref{l12} give
\begin{equation*}
\lim_{s\rightarrow 1}g'''(s)=\left\{\begin{array}{cl}
+\infty,&\mathrm{if}\,\,\,\beta_0<1,\\[2ex]
-\infty,&\mathrm{if}\,\,\,\beta_0\geq 1,
\end{array}\right.
\,\,\,\,g'''(k)=-\dfrac{25(23\kappa+31\beta_0-54)}{3888(\kappa-1)^3},
\end{equation*}
which implies the  assertions of this proposition.
\end{proof}

\begin{corollary}\label{c18}
Let $s\in(1,\kappa)$ and $\beta_1=1$.
\begin{itemize}
\item[(a)] If $\beta_0\in(-\infty,(23\kappa-54)/31]\cup[1,+\infty)$, then $g(s)$ has at most one inflection
points.
\item[(b)] If $\beta_0\in((23\kappa-54)/31,1)$, then $g(s)$ has at most two inflection
points.
\end{itemize}
\end{corollary}
\begin{proof}
Note that the zero of $g'''(s)$ is the maximum or minimum point of $g''(s)$ and $g(\kappa)=0$. The assertions of this
proposition follows from Proposition \ref{p17}.
\end{proof}

\begin{proof}[Proof of Theorem \ref{t14}]

Firstly  we note that $g(\kappa)=0$.

If $\beta_1=1$, then the statement (a) and (b) follows from Corollary \ref{c18}.

If  $\beta_1=0,\beta_0\not=0$, then   Lemma \ref{l16} shows that
$g'''(s)=\beta_0w'''(s)\not=0$, which implies that $g(s)$ has at
most one inflection point. This yields the statement (c).

If $\beta_0=\beta_1=0$, then $g(s)=P_2(s)=(s-\kappa)(\alpha_1+\kappa\alpha_2+\alpha_2 s)$. The assertion (d) follows.
\end{proof}

In the end of this section, we prove Theorem \ref{t1}.
\begin{proof}[Proof of Theorem \ref{t1}] Proposition \ref{p13} and Theorem \ref{t14} show  that $I(s)$ has at most
five zeros in $(1,\kappa)$. This implies that the perturbed system \eqref{eq2} has at most five limit cycles which emerge from the period annulus
around the center. The second assertion of Theorem \ref{t1} follows from Corollary \ref{c9}.
\end{proof}

\section{Comments}

Zoladek  conjectured that the exact upper bound of the cyclicity  of
the period annulus for $Q_4$ is three\cite{I1,Zo}.  Unfortunately
we can not
 prove Zoladek's conjecture in this paper.

As we mentioned before, the argument principle gives a shorter proof of Theorem \ref{t1}. However it seems clear to us that it
does not allow to go further in Zoladek's conjecture. Our approach is perhaps more involved from the computational point of view,
 but we think that it may  provide a way to attack the problem in the future paper. For instance, we can get the following results from
 Lemma \ref{l15}, Lemma \ref{l16} and Corollary \ref{c18}.

1.  If $\beta_1=1,\,\,\beta_0<1,\,P_2(\beta_0)\leq 0$ (resp. $\beta_0>\kappa,\,P_2(\beta_0)\geq 0$), then
$(P_2(s)/Q_1(s))''+w''(s)=P_2(\beta_0)/(s-\beta_0)^3+w''(s)<0$. This implies that $g(s)$ has
at most one zero in $(1,\kappa)$. Therefore $I(s)$ has three zeros in the same interval.

2. If  $\beta_1=1,\,\,\beta_0<1,\,0<P_2(\beta_0)<25(1-\beta_0)/(432(\kappa-1)^2)$, then
$(P_2(s)/Q_1(s))''+w''(s)=P_2(\beta_0)/(s-\beta_0)^3+w''(s)\leq 0$. This yields $I(s)$ has three zeros in $(1,\kappa)$.

3. If
$\beta_1=1,\,\,\beta_0\in(-\infty,(23\kappa-54)/31]\cup[1,+\infty),\,\,P_2(1)g'(\kappa)>0$,
then $g(s)$ has at most one zero. Hence $I(s)$ has at most three
zeros.

Here we just list the  partial results we have proved. We wish that
the above results will be helpful for proving Zoladek's conjecture
in the future paper.


\begin{thebibliography}{99}
\bibitem{A}
 V.I.Arnold, {\it Geometrical methods in the theory of differential equations}, Springer-Verlag, New York, 1983.
\bibitem{B}
{\sc N.N. Bautin}, {\it On the number of limit cycles which appear with the variation of the coefficients from
an equilibrium position of focus or center type}, Amer. Math. Soc. Translations {\bf 100}(1954), 1--19.
\bibitem{CJ}
{\sc C.Chicone and M.Jacobs}, {\it Bifurcation of limit cycles from quadratic isochronous}, J.
Differential Equations {\bf 91}(1991), 268--326.
 \bibitem{CLLZ}
{\sc F. Chen, C. Li, J. Llibre and Z.Zhang}, {\it A unified proof on the weak Hilbert 16th problem
for $n=2$}, J. Differential Equations {\bf 221}(2006), 309--342.
\bibitem{DLA}
{\sc F. Dumortier, J. Llibre, J. C. Art\'{e}s}, {\it Qualititive
theory of planar differential systems}, Springer, 2006.
\bibitem{DLZ}
 {\sc F. Dumortier, C. Li and Z. Zhang}, {\it Unfolding of a
 quadratic integrable system with two centers and two unbounded
 heteroclinic loop}, J. Differential Equations {\bf 139}(1997),
 146--193.
\bibitem{G}
{\sc L. Gavrilov}, {\it The infinitesimal 16th Hilbert problem in the quadratic case}, Invent. math. {\bf 143}(2001), 449--497.
\bibitem{GI}
{\sc L. Gavrilov and I.D. Iliev}, {\it quadratic perturbations of quadratic codimension-four centers}, J. Math, Anal. Appl. {\bf 357}(2009), 69--76.
\bibitem{GGI}
{\sc S. Gautier, L. Gavrilov, I. D. Iliev}, {\it Perturbations of quadratic centers of genus one},
Discrete Contin. Dyn. Syst. {\bf 25}(2009), 511--535.
\bibitem{GMV}
M. Grau, F. Manosas, J. Villadelprat, {\it  A Chebyshev criterion
for Abelian integrals}, to appear in Trans. Amer. Math. Soc.
\bibitem{H}
{\sc P.Hartman}, Ordinary differential equations (the second edition), Birkh$\ddot{a}$user, Boston$\cdot$Basel$\cdot$Stuttgart, 1982.
\bibitem{HI1}
{\sc E. Horozov, I.D. Iliev}, {\it On saddle-loop bifurcation of limit cycles in perturbations of quadratic
Hamiltonian systems}, J. Differential Equations {\bf 113}(1994), 84--105.
\bibitem{HI2}
{\sc E. Horozov, I.D. Iliev}, {\it On the number of limit cycles in perturbations of quadratic Hamiltonian
systems}, Proc. London Math. Soc. {\bf 69}(1994), 198--224.
\bibitem{I1}
{\sc I.D. Iliev}, {\it Perturbations of quadratic centers}, Bull. Sci. Math. {\bf 122}(1998), 107--161.
\bibitem{I2}
{\sc I.D. Iliev}, {\it The cyclicity of the period annulus of the quadratic Hamiltonian triangle}, J. Differential
Equations {\bf 128}(1996), 309--326.
\bibitem{ILY}
{\sc I.D.Iliev, C. Li and J. Yu}, {\it Bifurcations of limit cycles from quadratic non-
Hamiltonian systems with two centres and two unbounded heteroclinic loops}, Nonlinearity
{\bf 18}(2005), 305--330.
\bibitem{Li}
{\sc Weigu Li}, {\it  Normal form Theory and its application}(in Chinese), Since Press, 2000.
\bibitem{LR}
{\sc C. Li and R. Roussarie}, {\it The cyclicity of the elliptic
segment loops of reversible quadratic Hamiltonian systems under
quadratic perturbations}, J. Differential Equations {\bf 205}(2004),
488--520.
\bibitem{LZ}
{\sc Haihua Liang and Yulin Zhao}, {\it On the period function of
reversible quadratic centers with their orbits inside quartics},
Nonlinear Analysis {\bf 71}(2009), 5655--5671.
\bibitem{M}
{\sc P. Mardesic}, {\it Chebyshev systems and the versal unfolding of the cusps of order $n$}, Travaux en Cours,
Hermann, Paris, {\bf 57}(1998), pp. 153.
\bibitem{P}
{\sc G. S. Petrov}, {\it Nonoscillation of elliptic integrals}(in Russian), Funct. Anal. Appl. {\bf 24}(1990), No.3, 45-50.
\bibitem{R}
{\sc R. Roussarie}, {\it Bifurcation of planar vector fields and Hilbert's sixteenth problem}, Birkh$\ddot{a}$user Verlag, Basel$\cdot$Boston$\cdot$Berlin,
1998.
\bibitem{ZLL}
{\sc Y. Zhao, Z.Liang and G. Lu}, {\it The cyclicity of the period annulus of the quadratic
Hamiltonian systems with non-Morsean point}, J. Differential Equations {\bf 162}(2000),
199--223.
\bibitem{ZZ}
{\sc Y. Zhao and Siming Zhu}, {\it Perturbations of the non-generic quadratic Hamiltonian vector fields with hyperbolic segment}, Bull. Sci. math.
{\bf 125}(2001), no. 2, 109--138.

\bibitem{Zo}
{\sc H. Zoladek}, {\it Quadratic systems with center and their perturbations}, J. Differential Equations
{\bf 109}(1994),  223--273.
\end{thebibliography}
\end{document}